\documentclass[12pt,draftcls,onecolumn]{IEEEtran}

\usepackage{amsmath} 
\usepackage{amssymb}  
\usepackage{mathtools}
\usepackage{cite}

\usepackage{color}
\usepackage[colorlinks=true,linkcolor=blue]{hyperref}

\newcommand{\matt}[1]{\begin{bmatrix}#1\end{bmatrix}}
\newcommand{\m}{\mathbf{m}}

\def\scr#1{{\mathcal{#1}}}
\newcommand{\R}{\mathbb{R}}
\newcommand{\cF}{\mathcal{F}}

\newcommand{\cN}{\mathcal{N}}
\newcommand{\cC}{\mathcal{C}}
\newcommand{\bN}{\mathbb{N}}
\def\eq#1{\begin{equation}#1\end{equation}}
\def\rep#1{(\ref{#1})}

\newtheorem{theorem}{Theorem}
\newtheorem{lemma}{Lemma}

\newtheorem{proposition}{Proposition}

\newcommand{\dfb}{\stackrel{\Delta}{=}}
\def\qed{ \rule{.08in}{.08in}}

\newenvironment{proof-of}[1]
  {\noindent\textbf{Proof of #1: }}
  {\hspace*{\fill}~\qed\par\endtrivlist\unskip}

\title{\Large \bf A Distributed Algorithm for Computing a Common Fixed Point of a Finite Family of Paracontractions
\thanks{%
This work was supported by National Science Foundation grant n.~1607101.00 and US Air Force grant n.~FA9550-16-1-0290.
Daniel Fullmer and A. Stephen Morse are with the Department of Electrical Engineering, Yale University,~\texttt{\{daniel.fullmer,as.morse\}@yale.edu}.
}}

\author{Daniel Fullmer
  \hspace*{0.5 in}
  A. Stephen Morse
}

\begin{document}

\maketitle
\thispagestyle{empty}
\pagestyle{empty}

\begin{abstract}
A distributed algorithm is described for finding a common fixed point of a family of $m>1$  nonlinear maps
 $M_i:\R^n\rightarrow\R^n$
assuming that each map is a paracontraction and that at least one such common fixed point exists.
The common fixed point is simultaneously computed by $m$ agents
assuming each agent $i$ knows only $M_i$,
the current estimates of the fixed point generated by its neighbors, and nothing more.
Each agent recursively updates its estimate of a fixed point by utilizing the current estimates generated by each of its neighbors.
Neighbor relations are characterized by a time-varying directed graph $\bN(t)$.
It is shown under suitably general conditions on $\bN(t)$,
that the algorithm causes all agent estimates to converge to the same common fixed point of the $m$ nonlinear maps.
\end{abstract}

\section{Introduction}
This paper is concerned with the development of a distributed algorithm for enabling a group of $m>1$
 autonomous agents to solve certain types of nonlinear equations over a time-varying network.
The type of equations to which we are referring are described by the system
\begin{equation}\label{pc-sys}
M_i(x) = x, \ i \in \m
\end{equation}
where $\m \dfb \{1, 2, \dots, m\}$ and $M_i : \R^n \to \R^n$, $i \in\m$.
It is assumed that at least one solution to~(\ref{pc-sys}) exists,
\{i.e., the $M_i$ have at least one common fixed point\}
and that for $i \in \m$, agent $i$ knows $M_i$.
Each agent $i$ has a time dependent state vector $x_i(t)$ taking values in $\R^n$,
which is its estimate of a common fixed point.
It is assumed that each agent can receive information from its neighbors.
Specifically, agent $i$ receives the vector $x_j(t)$ at time $t$ if agent $j$ is a neighbor of agent $i$ at time $t$.
We write $\cN_i(t)$ for set of the labels of agent $i$'s neighbors at time $t$,
and we always take agent $i$ to be a neighbor of itself.
Neighbor relations at time $t$ can be conveniently characterized by a directed \textit{neighbor graph}
$\bN(t)$ with $m$ vertices and a set of arcs defined so that there is
an arc in $\mathbb{N}(t)$ from vertex $j$ to vertex $i$ just in case agent $j$ is a neighbor of agent $i$ at time $t$.
As each agent is a neighbor of itself, the neighbor graph $\bN(t)$ has self-arcs at each vertex.
In general terms, the problem of interest is to develop algorithms, one for each agent,
which will enable all $m$ agents to iteratively compute a common fixed point of all of the $M_i$.
This paper focuses on nonlinear maps which are ``paracontractions.''
A continuous nonlinear map $M:\R^n\rightarrow\R^n$ is a \textit{paracontraction}
with respect to a given norm $\|\cdot\|$ on $\R^n$,
if $\|M(x)-y\| < \|x-y\|$ for all $x\in\R^n$ satisfying $x \neq M(x)$ and all $y \in \R^n$ satisfying $y = M(y)$~\cite{paracontract}.
In most applications, paracontractions have multiple fixed points.
The concept of a paracontraction has been used in a system theoretic framework to study variants of the classical consensus problem~\cite{Xiao2006,Wu2007,Antsaklis2008}.

Motivation for this problem stems, in part, from~\cite{lineareqn}
which deals with the problem of devising a distributed algorithm for finding a solution to the linear equation
$Ax = b$ assuming the equation has at least one solution and agent $i$ knows a pair of the matrices
$(A_i^{n_i\times n},b_i^{n_i\times 1})$
where $A = \matt{A_1' &A_2' &\cdots & A_m'}'$ and $b = \matt{b_1'&b_2'&\cdots & b_m'}'$.
Assuming each $A_i$ has linearly independent rows,
one local update rule for solving this problem in discrete time is of the form
\eq{x_i(t+1) = L_i(\bar{x}_i(t)), \quad t=1,2,3\ldots\label{eqn:linear}}
where $L_i:\R^n\rightarrow\R^n$ is the affine linear map $x\longmapsto x - A_i' (A_i A_i')^{-1} (A_i x - b_i)$,
\eq{\bar{x}_i(t) = \frac{1}{m_i(t)}\sum_{j\in\scr{N}_i(t)} x_j(t) \quad t=1,2,3,\ldots\label{eqn:linear-xhat}}
and $m_i(t)$ is the number of labels in $\scr{N}_i(t)$~\cite{ACC16.1}.
The map $L_i$ is an example of a paracontraction  with respect to the 2-norm on $\R^{n}$.
To understand why this is so, note that $y=L_i(y)$ if and only if $A_i y = b_i$ and for any such $y$,
$L_i(x) - y = P_i(x-y)$ where $P_i$ is the orthogonal projection matrix $P_i = I-A_i'(A_iA_i')^{-1}A_i$.
For any $y$ satisfying $L_i(y)=y$, the inequality $x\neq L_i(x)$ is equivalent
to $x-y\notin\ker A_i$ and $\ker A_i = {\rm image }\;P_i$ so
$x-y\notin{\rm image }\;P_i$ whenever $x\neq L_i(x)$ and $y\in{\rm image }\;P_i$.
But for such $x$ and $y$, $\|P_i(x-y)\|_2<\|x-y\|_2$.
Since $L_i(x) - y = P_i(x - y)$, $L_i$ is a paracontraction as claimed.

There are many other examples of paracontractions discussed in the literature~\cite{paracontract,Byrne2007}.
Each of the following examples is a paracontraction with respect to the 2-norm on $\R^n$.
\begin{enumerate}
    \item The orthogonal projector $x \longmapsto \arg\min_{y \in \cC} \|x - y\|_2$ associated with a nonempty
        closed convex set $\cC$.
        This been used for a number of applications including the constrained consensus problem in \cite{Nedic2010}.
        The fixed points of this map are vectors in $\cC \subset \R^n$.
    \item The gradient descent map $x \longmapsto x - \alpha \nabla f(x)$ where $f :  \R^n \longrightarrow
     \R$ is convex and differentiable, $\nabla f$ is Lipschitz
     continuous with parameter $\lambda>0$, and $\alpha$ is a constant satisfying $0
     < \alpha < \frac{2}{\lambda}$.
     The fixed points of this map are vectors in $\R^n$ which minimize $f$~\cite{Ryu-primer}.
    \item The proximal map $x \longmapsto \arg\min_{y \in \cC} f(y) + \frac{1}{2} \| x - y \|_2$.
        associated with a closed proper convex function $f : \R^n \to (-\infty,\infty]$.
        The fixed points of this map are vectors in $\R^n$ which minimize $f$.
        See \cite{Eckstein1992} as well as \cite{Parikh2014}.
    \item Any `averaged' map, which is defined as a map $M\ :\ \R^n \to \R^n$ which satisfies $M(x) = \alpha N + (1-\alpha) x$ for all $x \in \R^n$,
        where $N$ is some map which is nonexpansive with respect to $\|\cdot\|_2$ and $0 < \alpha < 1$.
        If a map $M$ is averaged with parameter $\alpha$, we say that $M$ is $\alpha$-averaged~\cite{Bauschke2011}.
        In fact, the three examples above are all `averaged' maps~\cite{Ryu-primer}.
\end{enumerate}

\section{Paracontractions}
In this section, we review several basic properties of paracontractions.
Perhaps the most important is the following well-known theorem published in~\cite{paracontract}.
\begin{theorem}\label{elsner}
Suppose $\mathcal P$ is a finite set of paracontractions with respect to some given norm on $\R^n$.
Suppose that all of the paracontractions share at least one common fixed point.
Suppose that $P_1, P_2, \ldots$ is a sequence of paracontractions from $\mathcal P$.
Then the state $x(t)$ of the iteration
$$x(t+1) = P_{t}(x(t)), \quad t\in\{1,2,\ldots\}$$
converges to a point which is a common fixed point of those paracontractions which occur in the sequence infinitely often.
\end{theorem}
A number of classical results may be easily established by straightforward application of this theorem,
such as the convergence proof for the method of alternating \{or cyclic\} projections~\cite{paracontract}.

Below, certain useful propositions associated with paracontractions are described.
The proofs of these propositions may be found in the appendix.
Similar propositions can also be found in Chapter~4 of~\cite{Bauschke2011}.
In the following, the set of fixed points of a map $P\ :\ \R^n \to \R^n$ is denoted by $\cF(P) = \{ x\ :\ x = P(x) \}$.
Additionally the composition of two maps $P_1\ :\ \R^n \to \R^n$ and $P_2\ :\ \R^n \to \R^n$ is denoted by $P_1 \circ P_2$.
\begin{proposition}\label{pc-comp}
    Suppose $P_1\ :\ \R^n \to \R^n$ and $P_2\ :\ \R^n \to \R^n$ are each paracontractions with respect to same norm $\|\cdot\|$.
    Suppose $P_1$ and $P_2$ share at least one common fixed point, or in other words, $\cF(P_1) \cap \cF(P_2) \neq \emptyset$.
    Then the composition $P_1 \circ P_2$ is a paracontraction with respect to $\|\cdot\|$.
    Moreover, $\cF(P_1 \circ P_2) = \cF(P_1) \cap \cF(P_2)$.
\end{proposition}
It turns out that the set of fixed points of a paracontraction must be both closed and convex:
\begin{proposition}\label{pc-closed-convex}
Suppose $P\ :\ \mathbb R^n \to \mathbb R^n$ is a paracontraction.
Then $\cF(P) = \{ x\ :\ x = P(x) \}$ is closed and convex.
\end{proposition}

Recall that for a paracontraction $P\ :\ \R^n \to \R^n$,
it must be the case that $\|P(x) - y\| < \|x - y\|$ for all $x \notin \cF(P)$ and all $y \in \cF(P)$.
This property is referred to by a number of different names throughout the literature, such as `strictly quasi-nonexpansive.'
Our previous definition of a paracontraction also requires that the map be continuous.
So, a paracontraction is a continuous, strictly quasi-nonexpansive map.
One obvious consequence of the property above is that
$\|P(x)-y\| \leq \|x-y\|$ for all $x\in\R^n$ and all $y \in \cF(P)$.
Maps which satisfy this condition are called \textit{quasi-nonexpansive}.
So, any map which is a paracontraction must also be quasi-nonexpansive.
This fact will prove useful in the analysis to follow.
\begin{proposition}\label{linear-qne-pc-iff}
    Suppose $P\ :\ \R^n \to \R^n$ is a linear map and $\|\cdot\|$ is some norm on $\R^n$.
    Then, $P$ is quasi-nonexpansive with respect to $\|\cdot\|$
    if and only if it is nonexpansive with respect to $\|\cdot\|$.
    Moreover, $P$ is a paracontraction with respect to $\| \cdot \|$
    if and only if $\| P(x) \| < \| x \|$ for any $x \notin \cF(P)$.
\end{proposition}

\section{The Problem and Main Result}
The specific problem to which this paper is addressed is this.
Let $M_1,M_2,\ldots,M_m$ be a set of $m$ paracontractions with respect to the same norm $\|\cdot\|$.
Suppose that all of the paracontractions share at least one common fixed point.
Find conditions so that the $m$ states of the iterations
\begin{equation}\label{eqn:iter}
x_i(t+1) = M_i \left( \frac{1}{m_{i}(t)} \sum_{j\in\cN_i(t)} x_j(t) \right),\;\;i\in\mathbf{m},\;t\ge 1
\end{equation}
all converge to the same point as $t\rightarrow\infty$,
and that point is a common fixed point of the $M_i,\ i \in \m$,
where $m_i(t)$ and $\cN_i(t)$ are as defined earlier.

To state the main result of this paper, it is necessary to define certain concepts for sequences of directed graphs.
To begin, we write $\mathcal{G}$ for the set of all directed graphs with $m$ vertices.
By the \textit{composition} of two directed graphs $\mathbb{G}_p \in \mathcal{G}$ and $\mathbb{G}_q \in \mathcal{G}$ with the same vertex set,
written $\mathbb{G}_q\circ\mathbb{G}_p$, is meant that directed graph with the same vertex set and arc set defined so that
$(i, j)$ is an arc in the composition whenever there is a vertex $k$ such that
$(i, k)$ is an arc in $\mathbb{G}_p$ and $(k, j)$ is an arc in $\mathbb{G}_q$.
The definition of graph composition extends unambiguously to any finite sequence of directed graphs with the same vertex set.
We say that an infinite sequence of graphs $\mathbb G_1, \mathbb G_2, \dots$ in $\mathcal G$ is \textit{repeatedly jointly strongly connected},
if for some finite positive integers $l$ and $\rho_0$ and each integer $k > 0$,
the composed graph $\mathbb G_{kl + \rho_0 - 1} \circ \mathbb G_{kl + \rho_0 - 2} \circ \cdots \circ \mathbb G_{(k-1)l + \rho_0}$ is strongly connected.
The main result of this paper is as follows:

\begin{theorem}\label{main}
Let $M_1,M_2,\ldots, M_m$ be a set of paracontractions with respect to the $p$-norm $\|\cdot\|_p$ on $\R^n$ (for some $p$ satisfying $1 < p < \infty$).
Suppose the maps $M_1,M_2, \ldots, M_m$ share at least one common fixed point.
Suppose also that the sequence of neighbor graphs $\bN(1), \bN(2),\ldots $ is repeatedly jointly strongly connected.
Then the states $x_i(t)$ of the $m$ iterations defined by \rep{eqn:iter} all converge to the same point as $t\rightarrow\infty$,
and this point is common fixed point of the $M_i,\ i \in \m$.
\end{theorem}

A result similar to \autoref{main} was previously described in~\cite{nolcos-paracontract},
but required each $\bN(t)$, $t \ge 1$ to be strongly connected.
This paper extends that result to sequences of neighbor graphs which are repeatedly jointly strongly connected,
and presents a special case for which the convergence analysis is simple and instructive.

From the analysis which follows, it will be obvious that this result also applies to iterations of the more general form
\begin{equation}\label{eqn:convex-iter}
x_i(t+1) = M_i \left( \sum_{j\in\cN_i(t)} s_{ij}(t) x_j(t) \right),\;\;i\in\mathbf{m},\;t\ge 1
\end{equation}
where $s_{ij}(t)$ are nonnegative real-valued weights from a finite set,
and for each $i \in \m$ and $t \ge 1$,
$\sum_{j\in\cN_i(t)} s_{ij}(t) = 1$
and $s_{ij}(t) > 0$ if $j \in \cN_i(t)$ and $s_{ij}(t) = 0$ if $j \notin \cN_i(t)$.
As will be seen in the sequel,
the analysis which follows depends critically on there being only finitely many such weights.

It is interesting to note that the standard graphical condition for convergence of a consensus process~\cite{reachingp1},
namely `repeatedly jointly rooted,' is \underline{not} sufficient to ensure convergence for the problem considered in this paper.
Consider a simple counterexample in which $m=2$, $M_1$ and $M_2$ are orthogonal projectors onto two convex sets $\cC_1$ and $\cC_2$ which share a common point,
and the neighbor graphs $\bN(1), \bN(2), \ldots$ are all equal and have self arcs for agents 1 and 2 as well as an arc from agent 1 to agent 2.
Each neighbor graph in this sequence is rooted (as defined in~\cite{reachingp1}), and so the sequence is repeatedly jointly rooted.
Suppose the weights $s_{ij}(t)$ are constant $s_{11}(t) = 1$, $s_{12}(t) = 0$, $s_{21}(t) = 1/2$, $s_{22}(t) = 1/2$ for $t \ge 1$.
Suppose further $x_{1}(1) \in \cC_1$ but $x_{1}(1) \notin \cC_2$.
In this case, it is clear from (\ref{eqn:iter}) that $x_1(t+1) = M_1(x_1(t))$, $t \ge 1$,
but since $x_1(t) \in \cC_1$, it follows that $x_1(t+1) = x_1(t)$, $t \ge 1$.
Therefore, $x_1(t)$, $t \ge 1$ is constant and cannot converge to vector in $\cC_2$.
In fact, if the $M_i$ are the affine linear maps discussed in (\ref{eqn:linear})
then the repeatedly jointly strongly connected condition of \autoref{main} is actually necessary for convergence,
provided the convergence is required to be be exponential~\cite{lineareqn}.

The remainder of this paper is devoted to a proof of \autoref{main}.

\section{Analysis}
This section is organized as follows.
First, in \autoref{sec:combining}, the iterations (\ref{eqn:iter}) are written as a single iteration using stacked vectors in $\R^{mn}$.
In \autoref{sec:special}, \autoref{main} is shown under a special case for which the analysis is straightforward,
Finally, in \autoref{sec:general}, \autoref{main} is shown under the general case.

In \autoref{main}, the sequence of neighbor graphs is assumed to be repeatedly jointly strongly connected.
This condition is used in the proofs below to show that products of stochastic matrices meet certain conditions.
For an $m \times m$ matrix $A$ with nonnegative entries, we associate the $m$ vertex directed graph $\gamma(A)$ defined so that $(i,j)$ is an arc from $i$ to $j$ in the graph just in case the $ji$th entry of $A$ is nonzero.
Graph composition and matrix multiplication are closely related.
Indeed, composition is defined so that for any pair of nonnegative $m \times m$ matrices $A_1$, $A_2$, with graphs $\gamma(A_1)$, $\gamma(A_2) \in \mathcal G$, $\gamma(A_2 A_1) = \gamma(A_2) \circ \gamma(A_1)$.
The graph of the product of two nonnegative matrices $A_1,A_2 \in \R^{n\times n}$ is equal to the composition
of the graphs of the two matrices comprising the product.
In other words, $\gamma(A_2 A_1) = \gamma(A_2) \circ \gamma(A_1)$.

\subsection{Combined Iteration}\label{sec:combining}
To proceed, let us note that the family of $m$ iterations given by \rep{eqn:iter} may be written as a single iteration of the form
\eq{x(t+1) = M((S(t) \otimes I) x(t)), \quad t\ge 1\label{eqn:x-iter}}
where for any set of vectors $\{x_i\in\R^n,\;i\in\mathbf{m}\}$, $x\in\R^{mn}$ is the stacked vector
\begin{equation}\label{eqn:x-def}
    x = \matt{x_1 \\ x_2 \\ \vdots \\ x_m}
\end{equation}
$M:\R^{mn}\rightarrow\R^{mn}$ is the map
\begin{equation}\label{eqn:M-def}
    M(x) = \matt{M_1(x_1) \\ M_2(x_2) \\ \vdots \\ M_m(x_m)},
\end{equation}
$S(t)$ is an $m \times m$ stochastic matrix whose $ij$th entry is $s_{ij}(t) = 1/m_{i}(t)$ if $j \in \cN_i(t)$ and $s_{ij}(t) = 0$ if $j \notin \cN_i(t)$,
$I$ is the $n \times n$ identity matrix, and $S(t) \otimes I$ is the Kronecker product of $S(t)$ with $I$.

It is clear from the definition of $M$ that the set of fixed points of $M$ is
$\cF(M) = \{ \matt{x_1' & x_2' & \cdots & x_m'}' : x_i = M_i(x_i),\ i \in \m \}$.
In the sequel, we use $\cC \subset \mathbb R^{mn}$ to denote the \textit{consensus set},
$\cC = \{ \matt{x_1' & x_2' & \cdots & x_m'}' : x_i = x_j,\  i,j \in \mathbf m \}$.
Note that the intersection of these sets is
$\cF(M) \cap \cC = \{ \matt{x_1' & x_2' & \cdots & x_m'}' : x_i = x_j,\ x_i = M_j(x_i),\ i,j \in \m \}$.
In words, if $x$ is a vector in $\cF(M) \cap \cC$, then each of its subvectors are equal and each subvector is a common fixed point of the maps $M_1, M_2, \ldots, M_m$.
The set $\cF(M) \cap \cC$ is nonempty if the maps $M_1, M_2, \ldots, M_m$ share at least one common fixed point.
To prove convergence of the states $x_i(t)$ to the same common fixed point of the $m$ paracontractions,
it suffices to show that $x(t)$ converges to a vector in $\cF(M) \cap \cC$.

The analysis in the sequel will involve the use of \autoref{elsner} with maps which are shown to be paracontractions with respect to a `mixed vector norm', $\|\cdot\|_{p,q}$,
which we define for stacked vectors in $\R^{mn}$ of the form as in (\ref{eqn:x-def}).
For a norm $\|\cdot\|_p$ on $\R^n$ and a norm $\|\cdot\|_q$ on $\R^m$,
the \textit{mixed vector norm} $\|\cdot\|_{p,q}$ on $\R^{mn}$ is defined as follows:
\[ \| x \|_{p,q} = \left\lVert \begin{bmatrix}\|x_1\|_p \\ \|x_2\|_p \\ \vdots \\ \|x_m\|_p \end{bmatrix} \right\rVert_q \]
This is a `norm of norms,' first taking the $p$ norm of each subvector $x_i$, $i \in \m$, and then taking the $q$ norm of a vector consisting of those norms.

\subsection{Special Case}\label{sec:special}
With certain additional but somewhat restrictive assumptions, the proof of \autoref{main} turns out to be a simple application of \autoref{elsner}.
Toward this end, suppose each $M_i,\ i \in \m$ is a paracontraction with respect to $\|\cdot\|_2$
and each matrix $S(t),\ t \ge 1$ is doubly stochastic.
What makes this special case much simpler than the analysis for the general case is the fact that,
with these assumptions,
both $M$ and $S \otimes I$ are paracontractions with respect to $\|\cdot\|_{2,2}$, as shown below.

\begin{proposition}\label{M-pc-2}
    Suppose each map $M_1, \ldots, M_m$ is a paracontraction with respect to $\|\cdot\|_2$.
    Then the map $M$ as defined by (\ref{eqn:M-def}) is a paracontraction with respect to $\|\cdot\|_{2,2}$.
\end{proposition}
\begin{proof}
    First, note that $M$ is continuous since each $M_i,\ i \in \m$ is continuous.
    Next, suppose $x \notin \cF(M)$ and $y \in \cF(M)$.
    So $y_i \in \cF(M_i)$ for $i \in \m$, and there is some $j \in \m$ so that $x_j \notin \cF(M_j)$.
    Since each $M_i$, $i \in \m$ is a paracontraction with respect to $\|\cdot\|_2$,
    it follows that $\|M_i(x_i) - y_i\|_2 \le \|x_i - y_i\|_2$ for each $i \in \m$.
    Additionally, $\|M_j(x_j) - y_j\|_2 < \|x_j - y_j\|_2$
    since $x_j \notin \cF(M_j)$, $y_j \in \cF(M_j)$, and $M_j$ is a paracontraction.
    As a result, $\sum_{i \in \m} \|M_i(x_i) - y_i\|_2^2 < \sum_{i \in \m} \|x_i - y_i\|_2^2$.
    Consequently,
    \begin{equation}
        \|M(x) - y\|_{2,2} = \sqrt{\sum_{i\in\m} \|M_i(x_i) - y_i\|_2^2} < \sqrt{\sum_{i\in\m} \|x_i - y_i\|_2^2} = \|x - y\|_{2,2}
    \end{equation}
    Thus, $M$ is a paracontraction with respect to $\|\cdot\|_{2,2}$.
\end{proof}

\begin{proposition}\label{S-I-pc-2}
    Suppose $S = [s_{ij}]$ is an $m \times m$ doubly stochastic matrix with positive diagonal entries.
    Then $S \otimes I$ is a paracontraction with respect to $\|\cdot\|_{2,2}$.
\end{proposition}
\autoref{S-I-pc-2} is a simple consequence of \autoref{dbl-stochastic-S-pc}, which applies to all doubly stochastic matrices with positive diagonal entries.
\begin{lemma}\label{dbl-stochastic-S-pc}
    Suppose $S = [s_{ij}]$ is an $m \times m$ doubly stochastic matrix with positive diagonal entries.
    Then $S$ is a paracontraction with respect to $\|\cdot\|_2$.
\end{lemma}
\begin{proof}
    To begin, suppose the graph of $S$, $\gamma(S)$, contains $q \ge 1$ disjoint weakly connected components.
    Let $Q$ be a permutation matrix such that $S = Q(\textrm{diag}\{T_1, \ldots, T_q\}) Q'$,
    where each $T_i,\ i = 1, \ldots, q$ is an $m_i \times m_i$ doubly stochastic matrix with positive diagonal entries and a weakly connected graph,
    and $m_1, \ldots, m_q$ are positive integers such that $\sum_{i=1}^q m_i = m$.

    Note that $S'S = Q(\textrm{diag}\{T_1'T_1, \ldots, T_q'T_q\})Q'$.
    Since each $T_i,\ i = 1, \ldots, q$ is a doubly stochastic matrix with positive diagonal entries,
    each $T_i'T_i,\ i = 1, \ldots, q$ is also a doubly stochastic matrix with positive diagonal entries.
    Additionally, since each $\gamma(T_i),\ i = 1, \ldots, q$ is a weakly connected graph with self-arcs at each vertex,
    and $\gamma(T_i'T_i) = \gamma(T_i') \circ \gamma(T_i)$,
    it follows that each $\gamma(T_i'T_i),\ i = 1, \ldots, q$ is a strongly connected graph and self-arcs at each vertex.
    Thus, each $T_i'T_i,\ i = 1, \ldots, q$ is primitive, and so
    $Q(\textrm{diag}\{T_1'T_1, \ldots, T_q'T_q\})Q'$ must have an eigenvalue at $1$ of multiplicity $q$,
    and all other eigenvalues must have have magnitude less than $1$.
    As $S'S$ is similar to $Q(\textrm{diag}\{T_1'T_1, \ldots, T_q'T_q\})Q'$,
    it must also have an eigenvalue at $1$ of multiplicity $q$,
    and all other eigenvalues must have have magnitude less than $1$.

    Next, we claim that if $x \notin \cF(S)$ then $x \notin \cF(S'S)$.
    Suppose $x \in \cF(S'S)$, or in other words, $S'Sx = x$.
    It follows that $\textrm{diag}\{T_1'T_1, \ldots, T_q'T_q\} Q' x = Q' x$.
    Using the Perron-Frobenius Theorem,
    $Q'x = \begin{bmatrix}c_1 1' & c_2 1' & \cdots & c_q 1'\end{bmatrix}'$,
        where each $c_i,\ i = 1, \ldots, m$ is some real value and the corresponding $1$ is the vector of all ones in $\R^{m_i}$..
    As each $T_i,\ i = 1, \ldots, q$ is a stochastic matrix,
    it follows that $\textrm{diag}\{T_1, \ldots, T_q\} Q' x = Q' x$,
    and therefore, $Sx = x$, or in other words $x \in \cF(S)$.

    Now, suppose $x \notin \cF(S)$.
    From the previous claim, $x \notin \cF(S'S)$.
    Therefore, $S'Sx \neq x$ and so $x$ is not an eigenvector of $S'S$ associated with the eigenvalue 1.
    But recall that all other eigenvalues have magnitude less than 1.
    Consequently, $x'S'Sx < x'x$ as $S'S$ is symmetric.
    Thus, $\|Sx\|_2 < \|x\|_2$ and by \autoref{linear-qne-pc-iff}, $S$ is a paracontraction with respect to $\|\cdot\|_2$.
\end{proof}


\begin{proof-of}{\autoref{S-I-pc-2}}
    From the definition of $\|\cdot\|_{2,2}$ it is not difficult to see that for any vector $x \in \R^{mn}$, $\|x\|_{2,2} = \|x\|_{2}$.
    Additionally, $S \otimes I$ is a doubly stochastic matrix with positive diagonal entries.
    So, this proposition follows directly from \autoref{dbl-stochastic-S-pc}.
\end{proof-of}

With these preceding tools, we may now prove our main result in this special case.
Here, for simplicity, we also assume that each neighbor graph $\bN(t),\ t \ge 1$ is strongly connected.

\begin{proof-of}{\autoref{main} for Special Case}
Each neighbor graph in the sequence $\bN(1), \bN(2), \ldots$ has self arcs at each vertex because each agent is assumed to be a neighbor of itself.
So each matrix $S(t) \otimes I,\ t \ge 1$ must have positive diagonal entries.
From \autoref{M-pc-2} and \autoref{S-I-pc-2},
each $S(t) \otimes I$, $t \ge 1$ and $M$ is a paracontraction with respect to $\|\cdot\|_{2,2}$.
Since there is at least one common fixed point of the maps $M_1, \ldots, M_m$,
the maps $M$ and $S(t) \otimes I$, $t \ge 1$ must share at least one common fixed point.
By \autoref{pc-comp},
each $M \circ (S(t) \otimes I)$, $t \ge 1$ is a paracontraction
and $\cF(M \circ (S(t) \otimes I)) = \cF(M) \cap \cF(S(t) \otimes I)$ for each $t \ge 1$.
Note that there are only a finite number of such composed maps,
since the entries of each $S(t)$, namely $s_{ij}(t)$, may take only a finite number of possible values.
Applying \autoref{elsner} to the iteration defined by (\ref{eqn:x-iter}),
ensures that $x(t)$ will converge to a fixed point in the intersection of the sets of fixed points of those $M \circ (S(t) \otimes I)$ which occur infinitely often.
Since $\cF(M \circ (S(t) \otimes I)) = \cF(M) \cap \cF(S(t) \otimes I)$, the $x(t)$ must converge to a vector which is in $\cF(M)$ and also in $\cF(S(t) \otimes I)$ for those $S(t)$ which occur infinitely often.
However, under the assumption that each each neighbor graph $\bN(t),\ t \ge 1$ is strongly connected,
the graph of each $S(t)$ is strongly connected.
By the Perron-Frobenius Theorem, it follows that $\cF(S(t) \otimes I) = \cC$ for $t \ge 1$.
So, regardless of which particular $S(t)$ occur infinitely often, $x(t)$ must converge to a vector in $\cF(M) \cap \cC$.
\end{proof-of}

It is not difficult to relax the assumption that each neighbor graph $\bN(t),\ t \ge 1$ is strongly connected
to the more general condition that the sequence of neighbor graphs $\bN(1), \bN(t), \ldots$ is repeatedly jointly strongly connected, as in \autoref{main}.
The key step is to prove that the intersection of the sets of fixed points of those $S(t) \otimes I$ which occur infinitely often is just the consensus set, $\cC$.
In fact, this can be shown under a graphical condition even weaker than repeatedly jointly strongly connected.
The definition of repeatedly jointly strongly connected includes a uniformity condition which requires that, for some fixed finite number $l$, each successive composition of $l$ graphs is strongly connected.
Instead, the main result can be shown even if this number $l$ varies with each successive composition of graphs.

While the above proof is a straightforward application of \autoref{main},
it does not generalize to the case in which the $S(t)$ matrices are not doubly stochastic matrices.
Instead, it has proven necessary to reason about composed maps of sufficient length and prove that those maps are paracontractions with the requisite set of fixed points, as discussed in the following section.

\subsection{General Case}\label{sec:general}
There are two differences which make the proof of \autoref{main} more challenging than the analysis previously presented for the special case.
First, while the matrices $S(t)$ are stochastic,
they need not be doubly stochastic,
so \autoref{S-I-pc-2} may not apply.
In order to establish convergence with stochastic matrices without requiring these matrices be doubly stochastic,
it has proven useful to instead focus on the $\infty$-norm, or more precisely $\| \cdot \|_{p,\infty}$.
However, this leads to a second difficulty.
Unlike for the case with $\|\cdot\|_{2,2}$ as shown in \autoref{M-pc-2},
even if each $M_i,\ i \in \m$ is a paracontraction with respect to some norm $\|\cdot\|_p$,
$M$ need not be a paracontraction with respect to $\|\cdot\|_{p,\infty}$.

As an example, consider $m=2$, $x_1 \notin \cF(M_1)$, and $x_2 \in \cF(M_2)$ so that $x \notin \cF(M)$.
Suppose that $y_1 \in \cF(M_1)$ and $y_2 \in \cF(M_2)$ so that $y \in \cF(M)$.
Suppose also that $\|x_1 - y_1\|_p \le \|x_2 - y_2\|_p$,
so that
\begin{equation}\label{eqn:x2-max}
    \max_{i\in\m} \|x_i - y_i\|_p = \|x_2 - y_2\|_p
\end{equation}
Additionally,
$$\|M_1(x_1) - y_1\|_p < \|x_1 - y_1\|_p \le \| x_2 - y_2\|_p$$
since $M_1$ is a paracontraction.
Because $x_2 \in \cF(M_2)$, $x_2 = M_2(x_2)$ and consequently
$\|M_1(x_1) - y_1\|_p < M_2(x_2) - y_2$
From this, it follows that
\begin{equation}\label{eqn:M-x2-max}
    \max_{i\in\m} \|M_i(x_i) - y_i\|_p = \|M_2(x_2) - y_2\|_p
\end{equation}
Using (\ref{eqn:x2-max}), (\ref{eqn:M-x2-max}), and $x_2 = M_2(x_2)$ and it follows that
\begin{equation}
    \max_{i\in\m} \|M_i(x_i) - y_i\|_p = \max_{i \in \m} \| x_i - y_i \|_p
\end{equation}
Thus, $\|M(x) - y\|_{p,\infty} = \|x - y\|_{p,\infty}$.
So, in this case, $M$ is not a paracontraction with respect to $\|\cdot\|_{p,\infty}$.
However, the map $M$ is always quasi-nonexpansive in this norm.
\begin{proposition}\label{M-qne-infty}
    Suppose each $M_1, \ldots, M_m$ is a paracontraction with respect to $\|\cdot\|_p$.
    Let $M$ be the map as defined in~(\ref{eqn:M-def}).
    Then $M$ is quasi-nonexpansive with respect to $\|\cdot\|_{p,\infty}$.
\end{proposition}
\begin{proof}
    Suppose $x \in \R^{mn}$ and $y \in \cF(M)$.
    So $y_i \in \cF(M_i)$ for $i \in \m$.
    Since each $M_i$, $i \in \m$ is a paracontraction with respect to $\|\cdot\|_p$,
    it follows that $\|M_i(x_i) - y_i\|_p \le \|x_i - y_i\|_p$, $i \in \m$.
    As a result,
    \begin{equation}
        \max_{i\in\m} \|M_i(x_i) - y_i\|_p \le \max_{i\in\m} \|x_i - y_i\|_p
    \end{equation}
    and therefore $\|M(x) - y\|_{p,\infty} = \|x - y\|_{p,\infty}$.
    Thus, $M$ is a paracontraction with respect to $\|\cdot\|_{2,2}$.
\end{proof}

A similar result follows for stochastic matrices.
While the following results are stated for matrices of the form $S \otimes I$ with respect to $\| \cdot \|_{p,\infty}$,
by taking the dimension of the identity matrix $I$ to be $1 \times 1$,
these results also apply to general stochastic matrices with respect to $\| \cdot \|_\infty$.
\begin{proposition}\label{S-I-qne-infty}
    Suppose $S = [s_{ij}]$ is an $m \times m$ stochastic matrix.
    Then $S \otimes I$ is quasi-nonexpansive with respect to $\|\cdot\|_{p,\infty}$.
\end{proposition}
\begin{proof}
    Suppose $x \in \R^{mn}$.
    By the triangle inequality, $\| \sum_{j \in \m} s_{ij} x_j \|_p \le \sum_{j \in \m} s_{ij} \| x_j \|$ for each $i \in \m$.
    But since $S$ is stochastic, $\sum_{j \in \m} s_{ij} = 1$ for each $i \in \m$.
    Thus,
    $$\|(S \otimes I) x\|_{p,\infty} = \max_{i \in \m} \| \sum_{j \in \m} s_{ij} x_j \|_p \le \max_{i \in \m} \|x_i\|_p = \| x \|_{p,\infty}$$
    Therefore $\|(S \otimes I)x\|_{p,\infty} \le \|x\|_{p,\infty}$ for any $x \in \R^{mn}$, so $S$ is nonexpansive with respect to $\|\cdot\|_{\infty}$,
    By \autoref{linear-qne-pc-iff}, $S \otimes I$ is quasi-nonexpansive with respect to $\|\cdot\|_{p,\infty}$.
\end{proof}
\begin{proposition}\label{S-I-pc-infty}
    Suppose $S = [s_{ij}]$ is an $m \times m$ positive stochastic matrix.
    Then for any real value $p$ satisfying $1<p<\infty$,
    $S \otimes I$ is a paracontraction with respect to $\|\cdot \|_{p,\infty}$.
\end{proposition}
\begin{proof}
Since $S$ is positive, by Perron's Theorem, the set of fixed points of the map $x \longmapsto (S \otimes I)x$ is the consensus set, $\cC$.
Let $x = \matt{x_1'&x_2'&\cdots & x_m'}'$ be
any vector in $\R^{mn}$ which is not a fixed point of $S \otimes I$.
Then there must exist integers $i$ and $j$ such that $x_i\neq x_j$.
Suppose first that $x_i$ is a scalar multiple of $x_j$; i.e. $x_i = \lambda x_j$ for some scalar $\lambda$.
Without loss of generality assume $|\lambda|<1$, so $\|x_i\|_p< \|x_j\|_p$.
Clearly $\|x_i\|_p <\|x\|_{p,\infty}$ and for all $d \in \mathbf m$, $\|x_d\|_p \leq\|x\|_{p,\infty}$
Then for each $k \in \mathbf m$,
$$\begin{array}{rl}
    \left \|\sum_{d\in\mathbf{m}}s_{kd}x_d\right \|_p
    &\leq \sum_{d\in\mathbf{m}}\|s_{kd}x_d\|_p \\
    &= \sum_{d\in\mathbf{m}}s_{kd}\|x_d\|_p
    < \sum_{d\in\mathbf{m}}s_{kd}\|x\|_{p,\infty}.
\end{array}$$
This strict inequality holds because $S$ is positive, which ensures that $s_{ki} > 0$.
But $\sum_{d\in\mathbf{m}}s_{kd} = 1$ because $S$ is stochastic so
\eq{\left \|\sum_{d\in\mathbf{m}}s_{kd}x_d\right \|_p < \|x\|_{p,\infty}, \qquad k\in\mathbf{m}.\label{turk}}

Now suppose that $x_i$ is not a scalar multiple of $x_j$.
Then for each $k \in \mathbf m$, $s_{ki}x_i$ is not a scalar multiple of $s_{kj}x_j$.
By Minkowski's inequality, $\|s_{ki}x_i+s_{kj}x_j\|_p <\|s_{ki}x_i\|_p+\|s_{kj}x_j\|_p$
since $s_{ki}$ and $s_{kj}$ are both positive.
So
\eq{\|s_{ki}x_i+s_{kj}x_j\|_p < s_{ki}\|x_i\|_p+s_{kj}\|x_j\|_p, \quad k \in \mathbf{m}.\label{min}}
By the triangle inequality,
$$\left\|\sum_{d=1} s_{kd}x_d\right \|_p \leq \|s_{ki}x_i+s_{kj}x_j\|_p +\sum_{\substack{d\in\mathbf{m} \\ d \neq i,j}} \|s_{kd}x_d\|_p.$$
Thus using \rep{min},
$$\begin{array}{rl}
    \left \|\sum_{d\in \mathbf{m}} s_{kd}x_d\right\|_p
    &< \sum_{d\in\mathbf{m}}\|s_{kd}x_d\|_p
= \sum_{d\in\mathbf{m}}s_{kd}\|x_d\|_p \\
&\leq \sum_{d\in\mathbf{m}}s_{kd} \|x\|_{p,\infty}=\|x\|_{p,\infty}
\end{array}$$
so \rep{turk} holds for this case as well.
But
$$\|(S\otimes I)x\|_{p,\infty} = \max_{k\in\mathbf{m}}\left \|\sum_{d \in \mathbf{m}} s_{kd}x_d\right\|_p$$
so
\eq{\|(S\otimes I)x\|_{p,\infty}<\|x\|_{p,\infty}.\label{drop}}
So, from \autoref{linear-qne-pc-iff}, $S \otimes I$ is a paracontraction as claimed.
\end{proof}

The previous condition that $S$ be a positive stochastic matrix for $S \otimes I$ to
be a paracontraction is rather strong.
In a certain sense, this is a necessary condition as well.
(See Proposition~3.6 of~\cite{Mojskerc2014} for a related statement characterizing the complex-valued matrices which are paracontractions with respect to $\|\cdot\|_\infty$.)
\begin{proposition}\label{prop:positive-stochastic-sqne}
    Suppose $S = [s_{ij}]$ is an $m \times m$ stochastic matrix and assume that the set of fixed points of $S \otimes I$ is $\cC$.
    If $S \otimes I$ is a paracontraction with respect to $\|\cdot\|_{p,\infty}$, then $S$ is a positive matrix.
\end{proposition}
\begin{proof}
Suppose $S \otimes I$ is a paracontraction with respect to $\|\cdot\|_{p,\infty}$ and $\cF(S \otimes I) = \cC$.
Assume, to the contrary, that $S$ is not a positive matrix,
which means there must be indices $i, k \in \mathbf m$ such that $s_{ik} = 0$.
Consider the stacked vector $x \in \R^{mn}$ whose subvectors $x_j,\ j \in \m$ are given by
\[ x_j = \begin{cases} 0 & \textrm{if } j = k \\ z & \textrm{if } j \neq k \end{cases} \]
where $z$ is any vector in $\R^n$ such that $\|z\|_p = 1$.
Note that $x$ is not a fixed point of $S \otimes I$ since not all subvectors are equal, and so $x \notin \cC$.
Now, consider the $i$th subvector of $(S \otimes I)x$,
\[ \sum_{j \in \mathbf m} s_{ij} x_j = \sum_{j \neq k} s_{ij} z = z \sum_{j \in \mathbf m} s_{ij} = z. \]
Therefore,
$$\|(S \otimes I) x\|_{p,\infty} = \max_{i \in \m} \| \sum_{j \in \mathbf m} s_{ij} x_j \|_p = \max_{i \in \m} \| z \|_p = \|z\|_p $$
But, $\|x\|_{\infty} = 1$, so $||Sx||_{\infty} = ||x||_{\infty}$.
However, since $x \notin \cF(S \otimes I)$, from \autoref{linear-qne-pc-iff} this contradicts the assumption that $S \otimes I$ is a paracontraction with respect to $\|\cdot\|_{p,\infty}$.
\end{proof}

One approach to prove the main result would be to require that each $S(t)$ be a positive matrix.
This is far too restrictive as it would correspond to the requirement that each of the neighbor graphs $\bN(t),\ t \ge 1$ be a complete graph.
As will be seen in the sequel,
instead of showing that each individual $S(t) \otimes I$ is a paracontraction,
we instead show that composed maps of sufficient length are paracontractions.
The main technical lemma is as follows:
\begin{lemma}\label{composed-map-pc}
Let $M_i,\;i\in\mathbf{m}$ be a set of $m>1$ paracontractions with respect to the standard $p$ norm
$\|\cdot\|_p$ on $\R^n$ where $p$ is a real value satisfying $1<p<\infty$.
Let $S(1),S(2),\ldots, S(q)$ be a set of $q\geq 1$ $m\times m$ stochastic matrices.
If the $M_i,\;i\in\mathbf{m}$ have a common fixed point and the matrix product $S(q)S(q-1)\cdots S(1)$ is positive,
then the composed map
$\R^{mn}\rightarrow \R^{mn}$, $x \longmapsto ((S(q)\otimes I)M\circ\cdots \circ (S(1)\otimes I)M)(x)$
\begin{enumerate}
\item \label{jel1} is a paracontraction
with respect to the mixed vector norm $\|\cdot\|_{p,\infty}$.
\item \label{jel2} has as its set of fixed points  all stacked vectors of the form $\matt{y' & y'&\cdots & y'}'$
where $y$ is a common fixed point of the $M_i,\;i \in \mathbf{m}$.
In other words, its set of fixed points equals $\cF(M) \cap \cC$.
\end{enumerate}
\end{lemma}
The proof of \autoref{composed-map-pc} may be found in the sequel.
In the the proof of \autoref{main} will use \autoref{composed-map-pc} to show that a subsequence converges to a fixed point.
In order to show that the overall sequence also converges to a fixed point, we need one final lemma.
Define $\bar x(t) = (S(t) \otimes I)x(t)$ for each $t \ge 1$.
\begin{lemma}\label{subsequence-conv}
    Suppose $M$ and each $S(t) \otimes I$, $t \ge 1$ is quasi-nonexpansive with respect to some norm $\|\cdot\|$.
    If some subsequence of $\bar x(t)$ converges to $x^* \in \cF(M) \cap \cC$ as $t \to \infty$,
    then $x(t)$ also converges to $x^*$ as $t \to \infty$.
\end{lemma}
\begin{proof}
From (\ref{eqn:x-iter}),
$\| x(t+1) - x^* \| = \| M((S \otimes I) x(t)) - x^* \|$ for any $t \ge 1$.
Since $M$ is quasi-nonexpansive and $x^* \in \cF(M)$,
$ \| M((S \otimes I) x(t)) - x^* \| \le \| (S \otimes I)x(t) - x^* \| $ for any $t \ge 1$.
Since $S(t) \otimes I$ is quasi-nonexpansive and $x^* \in \cC \subset \cF(S(t) \otimes I)$,
$ \| (S \otimes I)x(t) - x^* \| \le \| x(t) - x^* \| $ for any $t \ge 1$.
In summary, for any $t \ge 1$,
\begin{equation}\label{eqn:x-fejer}
    \begin{aligned}
        \| x(t+1) - x^* \| &= \| M((S(t) \otimes I) x(t)) - x^* \| \\
                           &\le \| (S(t) \otimes I) x(t) - x^* \| \\
                           &\le \| x(t) - x^* \|
    \end{aligned}
\end{equation}

Let $\bar x(\rho_0), \bar x(\rho_1), \bar x(\rho_2), \ldots$ be a subsequence of $\bar x(t)$, $t \ge 1$
which converges to $x^* \in \cF(M) \cap \cC$.
From (\ref{eqn:x-fejer}), $\| x(t) - x^* \| \le \| x(\rho_k + 1) - x^* \|$ for any $k \ge 0$ and $t > \rho_k$.
Additionally, since $\bar x(\rho_k) = (S(\rho_k) \otimes I) x(\rho_k)$, it is also true from (\ref{eqn:x-fejer}) that $\| x(\rho_k + 1) - x^* \| \le \|  \bar x(\rho_k) - x^* \|$.
Thus $\| x(t) - x^* \| \le \| \bar x(\rho_k) - x^* \|$ for any $k \ge 0$ and $t > \rho_k$.
As a result, if $\lim_{k \to \infty} \|\bar x(\rho_k) - x^*\| = 0$ then $\lim_{t \to \infty} \|x(t) - x^* \| = 0$.
So, if the subsequence $\bar x(\rho_k),\ k \ge 0$ converges to some vector $x^* \in \cF(M) \cap \cC$, then the sequence $x(t)$ converges to this same vector.
\end{proof}

With the above results, it is now possible to prove our main result.

\begin{proof-of}{\autoref{main}}
Each neighbor graph in the sequence $\bN(1), \bN(2), \ldots$ has self arcs at each vertex because each agent is assumed to be a neighbor of itself.
By assumption, the sequence of neighbor graphs $\bN(1), \bN(2), \ldots$ is repeatedly jointly strongly connected,
so for some finite positive integers $l$ and $\rho_0$ and each integer $k > 0$,
the composed graph $\bN(kl + \rho_0 - 1) \circ \bN(kl + \rho_0 - 2) \circ \cdots \circ \bN((k-1)l + \rho_0)$ is strongly connected.
It is known that the composition of $q \dfb m-1$ self-arced, strongly connected graphs must be complete
\{c.f., Proposition 4 of \cite{reachingp1}\}.
Consequently, for each integer $k > 0$,
the composed graph $\bN(kql + \rho_0 - 1) \circ \bN(kql + \rho_0 - 2) \circ \cdots \circ \bN((k-1)ql + \rho_0)$ is complete.
Thus for each $k > 0$, the matrix $S(kql + \rho_0 - 1) S(kql + \rho_0 - 2) \cdots S((k-1)ql + \rho_0)$ is positive.

Define $z(k) = \bar x((k-1)ql + \rho_0 - 1)$ for each $k \ge 2$, so $z(k)$, $k \ge 1$ is a subsequence of $\bar x(t)$, $t \ge 1$.
Note that $z(k)$ is not defined for $k=1$, as that would imply $z(1) = \bar x(0)$ if $\rho_0 = 1$,
and $\bar x(0)$ is not defined.
From the definition of $\bar x(t)$ and $z(k)$, it follows that
\begin{equation}\label{eqn:z-iter}
    \begin{aligned}
        z(k+1) =& ((S(kql + \rho_0 - 1) \otimes I) M \circ \cdots \\
                & \quad \cdots \circ (S((k-1)ql + \rho_0) \otimes I) M)(z(k))
\end{aligned}
\end{equation}
for each $k \ge 2$.
It follows from Assertion~\ref{jel1} of \autoref{composed-map-pc} that the maps
$x \longmapsto ((S(kql + \rho_0 - 1)\otimes I)M\circ\cdots \circ (S((k-1)ql + \rho_0)\otimes I)M)(x)$, $k \ge 2$
are all paracontractions with respect to the mixed vector norm $\|\cdot \|_{p,\infty}$.
Note that there are only finitely many such maps,
since each map is a finite length composition,
and the entries of each $S(t)$, namely $s_{ij}(t)$, may take only a finite number of possible values.
Furthermore it is clear from Assertion~\ref{jel2} of \autoref{composed-map-pc},
that the set of fixed points of each map is $\cF(M) \cap \cC$.
It is clear from \autoref{elsner} and (\ref{eqn:z-iter})
that $z(k), k \ge 2$ must converge to such a fixed point $x^* \in \cF(M) \cap \cC$.
Since $z(k),\ k \ge 2$ is a subsequence of $\bar x(t),\ t \ge 1$,
using \autoref{M-qne-infty}, \autoref{S-I-qne-infty}, and \autoref{subsequence-conv}
it follows that $x(t)$ must also converge to this same vector $x^*$ in $\cF(M) \cap \cC$.
\end{proof-of}

In the sequel we develop the technical results needed to prove \autoref{composed-map-pc}.
In the proofs below, we will make use of the matrix
$\Phi(t,\tau) = \matt{\phi_{ij}(t,\tau)}_{m \times m}$ which we define as
$\Phi(t,\tau) = S(t)S(t-1)\cdots S(\tau+1)$ for $0 \le \tau < t \le q$
and 
$\Phi(t,t) = I$ for $0 \le t \le q$.
Note that
$S(t)\Phi(t-1,\tau) = \Phi(t,\tau) = \Phi(t,\tau+1)S(\tau+1),\;0\leq \tau<t\leq q$.
For each $i\in\mathbf{m}$, let $v_i(0)\in\R^n$ be an arbitrary but fixed vector, and define
\begin{equation}
  v_i(t+1) = \sum_{j\in\mathbf{m}}s_{ij}(t+1)M_j(v_j(t)),\quad 0 \le t < q.\label{v-defn}
\end{equation}
 We shall need the following lemmas.

\begin{lemma}\label{ineq}
Let $y^*$ be a common fixed point of the $M_i,\;i\in\mathbf{m}$.
For each $i\in\mathbf{m}$
\eq{\|v_i(t)-y^*\| \leq \sum_{j\in\mathbf{m}}\phi_{ij}(t,\tau) \|v_j(\tau) - y^*\|, \label{br1}}
for $0\leq \tau\leq t\leq q$.
\end{lemma}
\begin{proof}
Fix $0 \le \tau \le q$.
If $t = \tau$, then \rep{br1} holds for each $i\in\mathbf{m}$ since $\phi_{ij}(t,\tau) = 1$ whenever $i=j$ and $\phi_{ij}(t,\tau) = 0$ whenever $i \neq j$.

Suppose $t > \tau$ and \rep{br1} holds for some $t = \mu$ satisfying $\tau \le \mu < q$,
\eq{\|v_i(\mu)-y^*\| \leq \sum_{j\in\mathbf{m}}\phi_{ij}(\mu,\tau) \|v_j(\tau) - y^*\|,\quad i \in \mathbf{m}. \label{lem1-1}}
From \rep{v-defn} and the triangle inequality it follows that
$\|v_i(\mu+1)-y^*\|\leq \sum_{j\in\mathbf{m}}s_{ij}(\mu+1)\|M_j(v_j(\mu))-y^*\|$.
But the $M_i$ are paracontractions, so
\eq{\|v_i(\mu+1)-y^*\| \leq \sum_{j\in\mathbf{m}}s_{ij}(\mu+1)\|v_j(\mu)-y^*\|,\;i \in \mathbf{m}. \label{lem1-2}}
From \rep{lem1-1} and \rep{lem1-2}, it follows that
\begin{align*}
    \|v_i(\mu+1) -  y^*\| \!
    & \leq \! \sum_{k\in\mathbf{m}} s_{ik}(\mu+1) \! \sum_{j\in\mathbf{m}} \! \phi_{kj}(\mu,\tau)\|v_j(\tau)-y^*\| \\
    & = \! \sum_{j\in\mathbf{m}} \sum_{k\in\mathbf{m}} \! s_{ik}(\mu+1) \phi_{kj}(\mu,\tau) \|v_j(\tau)-y^*\|
\end{align*}
for each $i \in \mathbf m$.
But $\phi_{ij}(\mu+1,\tau) = \sum_{k\in\mathbf{m}} s_{ik}(\mu+1) \phi_{kj}(\mu, \tau)$ by the definition of $\Phi$,
so
$$\|v_i(\mu+1)-y^*\|\leq\sum_{j\in\mathbf{m}}\phi_{ij}(\mu+1,\tau) \|v_j(\tau)-y^*\|, \ i \in \mathbf m$$
which shows that \rep{br1} holds for $t=\mu+1$.
By induction, \rep{br1} holds for any $t$ satisfying $\tau < t \le q$.
Since $\tau$ was initially fixed, \rep{br1} holds for any $0 \le \tau \le t \le q$.
\end{proof}

\begin{lemma}\label{phi-ineq}
Let $y^*$ be a common fixed point of the $M_i,\; i\in\mathbf{m}$.
Then for each $i\in\mathbf{m}$,
\eq{\|v_i(q)-y^*\|\leq \sum_{j\in\mathbf{m}}\phi_{ij}(q,0)\|v_j(0)-y^*\|\label{gum1}}
and the following statements are true.
\begin{enumerate}
\item\label{gum2}
If there is a $t$ satisfying $0 \le t < q$ and a $j\in\mathbf{m}$
for which $\phi_{ij}(q,t)>0$ and $M_j(v_j(t))\neq v_j(t)$, then
\eq{\|v_i(q)-y^*\|< \sum_{p\in\mathbf{m}}\phi_{ip}(q,0)\|v_p(0)-y^*\|.\label{pup}}
\item\label{gum3}
If for every $t$ satisfying $0 \le t < q$ and $j\in\mathbf{m}$
it is true that $M_j(v_j(t)) = v_j(t)$
whenever $\phi_{ij}(q,t)>0$,
then
\eq{v_i(q)= \sum_{p\in\mathbf{m}}\phi_{ip}(q,0)v_p(0).\label{boop}}
\end{enumerate}
\end{lemma}

\begin{lemma}\label{class}
    If the matrix product $S(q)S(q-1)\cdots S(1)$ has a strongly connected graph,
    then
    $$\cF((S(q)\otimes I)M \circ \cdots \circ (S(1)\otimes I)M) = \cF(M) \cap \cC$$
    where $\circ$ denotes composition.
\end{lemma}
\begin{proof}
Let $x\in \cF(M)\cap\cC$.
Therefore $x\in\cC$ and all of the subvectors $x_i$ of $x = \matt{x_1'& x_2'&\cdots &x_m'}'$ must be equal.
This in turn implies that $(S(t)\otimes I)x = x,\;t\in\mathbf{q}$ since each $S(t)$ is a stochastic matrix.
Since $x\in\cF(M)$, $M(x) = x$.
Thus $(S(t)\otimes I)M(x) = x,\;t\in\mathbf{q}$ so $((S(q)\otimes I)M\circ \cdots \circ (S(1)\otimes I)M)(x) = x$.
Hence $x\in \cF((S(q)\otimes I)M\circ\cdots \circ (S(1)\otimes I)M )$
and thus $\cF(M)\cap\cC\subset \cF((S(q)\otimes I)M\circ\cdots \circ (S(1)\otimes I)M)$.

For the reverse inclusion, let $x\in \cF((S(q)\otimes I)M\circ\cdots \circ (S(1)\otimes I)M )$.
Set $v(0) = x$ and let $v(t)=\matt{v_1'(t) &v_2'(t) &\cdots & v_m'(t)}',\;0 \le t \le q$,
where $v_i(0) = x_i,\;i\in\mathbf{m}$ and for $t \in \mathbf{q}$, each $v_i(t)$ is as defined in \rep{v-defn}.
Then $v(q) = v(0) = x$.
Let $y^*$ be a common fixed point of the $M_i,\;i\in \mathbf{m}$.
In view of \rep{gum1},
$$\|v_i(q)-y^*\|\leq \sum_{j\in\mathbf{m}}\phi_{ij}(q,0)\|v_j(0)-y^*\|,\;i\in\mathbf{m}.$$
Thus
$w\leq \Phi(q,0)w$  where $\|v_i(q) -y^*\|$ is the $i$th component of the $n$-vector $w$ and $\leq$ here means component-wise.
Since $\Phi(q,0) = S(q)S(q-1)\cdots S(1)$ has a strongly connected graph, $\Phi(q,0)$ is irreducible.
It follows that $w = \Phi(q,0)w$ \{c.f., page 530 of \cite{Horn2013}\}.
By the Perron-Frobenius Theorem, all components of $w$ must be the same so all $\|v_i(q)-y^*\|, \; i\in\mathbf{m}$ must have the same value.

Suppose that for some $t$ satisfying $0 \le t < q$ and $i,j\in\mathbf{m}$, $\phi_{ij}(q,t)> 0$ and $M_j(v_j(t))\neq v_j(t)$.
 By Assertion~\ref{gum2} of \autoref{phi-ineq},
  $$\|v_i(q)-y^*\| <\sum_{p\in\mathbf{m}}\phi_{ip}(q,0)\|v_p(0)-y^*\|.$$
  Since $v(q) = v(0)$, it follows that $v_p(0) = v_p(q)$ and therefore,
$$\|v_i(q)-y^*\| <\sum_{p\in\mathbf{m}}\phi_{ip}(q,0)\|v_p(q)-y^*\|.$$
 Thus
$$\|v_i(q)-y^*\| <\sum_{p\in\mathbf{m}}\phi_{ip}(q,0)\|v_{a}(q)-y^*\|$$
 where $a\in\mathbf{m}$ is such that $\|v_{a}(q)-y^*\| =\max_{p\in\mathbf{m}}\|{v_p(q)-y^*}\|$.
Since $\sum_{p\in\mathbf{m}}\phi_{ip}(q,0) = 1$, $\|v_i(q)-y^*\| <\|v_{a}(q)-y^*\|$.
This contradicts the fact that all of the
  $\|v_i(q)-y^*\|, \; i\in\mathbf{m}$ have the same value.
Therefore for every $t$ satisfying $0 \le t < q$ and $j\in\mathbf{m}$,
it must be true that $M_j(v_j(t))= v_j(t)$ whenever $\phi_{ij}(q,t)>0$.

By hypothesis, the graph of $\Phi(q,0)$ is strongly connected so
for each $j\in\mathbf{m}$ there must be a $k\in\mathbf{m}$ such that $\phi_{kj}(q,0)> 0$.
This implies that $v_j(0)\in\cF(M_j),\;j\in\mathbf{m}$.
Therefore $x\in\cF(M)$.

Additionally, the hypothesis of Assertion~\ref{gum3} in \autoref{phi-ineq} is satisfied.
Therefore
 $$v_i(q) =  \sum_{p\in\mathbf{m}}\phi_{ip}(q,0)v_p(0),\;\;i\in\mathbf{m}.$$
Thus $v(q) = (S(q)\otimes I)\cdots(S(1)\otimes I) x$.
But $v(q) = v(0) = x$, so $x = ((S(q) \cdots S(1)) \otimes I)x$.
Since $S(q)\cdots S(1)$ is strongly connected,
the Perron-Frobenius Theorem ensures that all of the subvectors $x_i$ of $x = \matt{x_1'& x_2'&\cdots &x_m'}'$ must be equal and thus $x \in \cC$.
Therefore $\cF((S(q)\otimes I)\cdots(S(1)\otimes I)) \subset \cF(M) \cap \cC$.
\end{proof}

We now have the necessary lemmas to prove our main technical result.

\begin{proof-of}{\autoref{composed-map-pc}}
First, recall that $\cC = \{ \matt{y_1' & \cdots & y_m'}' : y_i = y_j,\  i,j \in \mathbf m \}$
and $\cF(M) = \{\matt{y_1' & \cdots & y_m'}' :\; y_i \in \cF(M_i),\; i \in \mathbf{m} \}$.
From this and \autoref{class} it follows that
$$\begin{array}{l}\cF((S(q) \otimes I)M \circ \cdots \circ (S(1) \otimes I)M) \\
\qquad \qquad = \{\matt{y' & \cdots & y'}' :\; y \in \bigcap_{i=1}^m \cF(M_i) \}\end{array}$$
Thus Assertion~\ref{jel2} of the theorem is true.

Pick $\bar{y}\in\cF((S(q)\otimes I)M\circ\cdots \circ (S(1)\otimes I)M)$
 and $x\notin\cF((S(q)\otimes I)M\circ\cdots \circ (S(1)\otimes I)M)$.
In view of \autoref{class}, either $x\notin\cF(M)$ or $x\notin\cC$;
moreover $\bar{y}\in\cF(M)$ and $\bar{y}\in\cC$.
Thus, $\bar{y}$ must be of the form $\bar{y} = \matt{y' &y'&\cdots &y'}$ for some vector $y\in\R^n$.
In addition, $y$ must be a common fixed point of the $M_i,\;i\in\mathbf{m}$.

Set $v_i(0) = x_i,\;i\in\mathbf{m}$ where $\matt{x_1' &x_2'&\cdots x_m'}' = x$
and let $v_i(t),\ t \in \mathbf{q}$ be as defined by \rep{v-defn}.
To complete the theorem's proof, it is sufficient to show that
if $v(0) \notin \cF(M)$ or $v(0) \notin \cC$, then
\eq{\|v_i(q) - y\|_p< \max_{j\in\mathbf{m}}\|v_j(0) - y\|_p,\;\;i \in \mathbf{m}.\label{bub}}
This is sufficient because \rep{bub} implies $\max_{j \in \mathbf{m}} \|v_j(q) - y\|_p < \max_{j \in \mathbf{m}} \|v_j(0) - y\|_p$,
and therefore $\|((S(q)\otimes I)M \circ \cdots (S(1) \otimes I)M)(v(0)) - y\|_{p,\infty} < \|v(0) - y\|_{p,\infty}$.

Fix $i \in \mathbf{m}$.
We claim that if there is a $t$ satisfying $0 \le t < q$ and a $j\in\mathbf{m}$ for which
$\phi_{ij}(q,t) > 0$ and $M_j(v_j(t)) \neq v_j(t)$ then \rep{bub} holds.
To justify this claim
note first  that
$$\sum_{j \in \mathbf{m} }\phi_{ij}(q,0) \|v_j(0) - y\|_p
\! \leq \!\! \left (\sum_{j \in \mathbf{m}} \phi_{ij}(q,0) \! \right ) \! \max_{j\in\mathbf{m}} \|v_j(0) - y\|_p.$$
But $\sum_{j \in \mathbf m} \phi_{ij}(q,0) = 1$ so
\eq{\sum_{j \in \mathbf{m}} \phi_{ij}(q,0) \|v_j(0) - y\|_p
\leq \max_{j\in\mathbf{m}} \|v_j(0) - y\|_p.\label{pioo}}
If there is a $t$ satisfying $0 \le t < q$ and a $j\in\mathbf{m}$
for which $\phi_{ij}(q,t) > 0$ and $M_j(v_j(t)) \neq v_j(t)$,
then by Assertion~\ref{gum2} of \autoref{phi-ineq}
$$\|v_i(q)-y\| < \sum_{j \in \mathbf{m}} \phi_{ij}(q,0) \|v_j(0) - y\|_p.$$
Since this and \rep{pioo} imply \rep{bub}, the claim is true.


To complete the proof there are two cases to consider, the first being  when  $v(0) \notin \cF(M)$.
In this case there is some $j \in \mathbf{m}$ such that $v_j(0) \notin \cF(M_j)$.
By hypothesis, $\Phi(q,0) = S(q) \cdots S(1)$ is positive and so $\phi_{ij}(q,0) > 0$.
Therefore with this value of $j$ and $t=0$,
$\phi_{ij}(q,t) > 0$ and $M_j(v_j(t)) \neq v_j(t)$.
Hence \rep{bub} holds in this case.

Now consider the case when $v(0) \notin \cC$.
Note that $\Phi(q,0) \otimes I$ is a paracontraction by \autoref{S-I-pc-infty} and the assumption that
$\Phi(q,0) = S(q) \cdots S(1)$ is a positive matrix.
Clearly
\eq{\|(\Phi(q,0) \otimes I) v(0) - \bar {y}\|_{p, \infty}
 < \|v(0) - \bar{y}\|_{p,\infty}.
    \label{S-strict}}
In other words
\eq{\max_{j\in\mathbf{m}} \left \|\sum_{k \in m} \phi_{jk}(q,0) v_k(0) - y \right \|_p
< \max_{j\in\mathbf{m}} \left \|v_j(0) - y \right \|_p. \label{S_i-strict}}

As noted in the above claim,
if there is a $t$ satisfying $0 \le t < q$ and a $j\in\mathbf{m}$
for which $\phi_{ij}(q,t) > 0$ and $M_j(v_j(t)) \neq v_j(t)$
then \rep{bub} holds.
If on the other hand, there is no $t$ satisfying $0 \le t < q$ and $j\in\mathbf{m}$
for which $\phi_{ij}(q,t) > 0$ and $M_j(v_j(t)) \neq v_j(t)$
then Assertion~\ref{gum3} of \autoref{phi-ineq} applies, and so
$$v_i(q)= \sum_{p\in\mathbf{m}}\phi_{ip}(q,0)v_p(0).$$
Therefore
$$\|v_i(q)-y\|_p = \left\|\sum_{p\in\mathbf{m}}\phi_{ip}(q,0)v_p(0) - y\right\|_p$$
Additionally,
$$\left\| \sum_{p\in\mathbf{m}}\phi_{ip}(q,0)v_p(0) -y\right\| \le \max_{j\in\mathbf{m}} \left \|\sum_{k \in m} \phi_{jk}(q,0) v_k(0) - y \right \|_p$$
 Finally, from this and \rep{S_i-strict}, it follows that \rep{bub} is true.
\end{proof-of}

\section{Concluding Remarks}
It may be possible to relax the condition that the sequence of neighbor graphs be `repeatedly jointly strongly connected' in \autoref{main} to a more general graphical condition which does not require a uniformity condition,
as discussed in the special case.
However, the proof technique for the general case used in this paper relied on showing that certain composed maps are paracontractions and may not be amenable to this extension.
It would also be interesting to determine necessary conditions on the sequence of neighbor graphs which ensure convergence.
Finally, the definition of a paracontraction seems to be too general as to establish meaningful convergence rates.
With this in mind, it may be useful to consider stronger versions of paracontractions for which convergence rate results may apply.

\bibliographystyle{unsrt}
\bibliography{journal}



\newpage
\section{Appendix}
\begin{proof-of}{\autoref{pc-comp}}
    Suppose $x \in \cF(P_1) \cap \cF(P_2)$.
    By definition, $x = P_1(x) = P_2(x)$ and so $x = P_1(P_2(x)) = (P_1 \circ P_2)(x)$.
    Therefore $\cF(P_1) \cap \cF(P_2) \subset \cF(P_1 \circ P_2)$.

    Next suppose $x \in \cF(P_1 \circ P_2)$.
    Let $y$ be a common fixed point of $P_1$ and $P_2$, i.e.~$y \in \cF(P_1)$ and $y \in \cF(P_2)$.
    Since $P_1$ and $P_2$ are both paracontractions and $x \in \cF(P_1 \circ P_2)$,
    \begin{equation}
        \| x - y \| = \| P_1(P_2(x)) - y \| \le \| P_2(x) - y \| \le \| x - y\|
    \end{equation}
    If, however, $x \notin \cF(P_2)$ then $\| P_2(x) - y \| < \| x - y \|$ and so $\|x - y\| < \|x - y\|$, which is a contradiction.
    So it must be the case that $x \in \cF(P_2)$.
    Therefore,
    \begin{equation}
        \| x - y \| = \| P_1(P_2(x)) - y \| = \| P_1(x) - y \| \le \| x - y \|
    \end{equation}
    Similarly, if $x \notin \cF(P_1)$, then $\| P_1(x) - y \| < \| x - y \|$ and so $\|x - y\| < \|x - y\|$, which is a contradiction.
    So $x \in \cF(P_1)$ as well as $x \in \cF(P_2)$,
    which implies $\cF(P_1 \circ P_2) \subset \cF(P_1) \cap \cF(P_2)$.

    To show that $P_1 \circ P_2$ is a paracontraction,
    Suppose $x \notin \cF(P_1 \circ P_2)$ and $y \in \cF(P_1 \circ P_2)$.
    We have previously shown $\cF(P_1 \circ P_2) = \cF(P_1) \cap \cF(P_2)$.
    As a result, $y \in \cF(P_1)$ and $y \in \cF(P_2)$.
    Additionally, either $x \notin \cF(P_1)$ or $x \notin \cF(P_2)$.
    We claim that $\| P_1(P_2(x)) - y \| < \| x - y \|$.
    If $x \notin \cF(P_2)$ then $\| P_1(P_2(x)) - y \| \le \| P_2(x) - y \| < \| x - y \|$.
    If instead $x \in \cF(P_2)$, it must be the case that $x \notin \cF(P_1)$,
    and so $\| P_1(P_2(x)) - y \| = \| P_1(x) - y \| < \| x - y \|$.
    Therefore, the claim is true.
    Since both $P_1$ and $P_2$ are paracontractions, they are continuous,
    and so their composition $P_1 \circ P_2$ is continuous as well.
    Consequently, $P_1 \circ P_2$ is a paracontraction.
\end{proof-of}
\begin{proof-of}{\autoref{pc-closed-convex}}
    To show that $\cF(P)$ is convex,
    suppose $x_1, x_2, \ldots,$ is a sequence of vectors in $\cF(P)$ which converges to a vector $x^*$,
    in other words, $\lim_{k \to \infty} x_k = x^*$.
    Since $P$ is a paracontraction, it is continuous, and thus $\lim_{k \to \infty} P(x_k) = P(x^*)$.
    But each $x_k$ is a fixed point of $P$, $x = P(x)$, so $\lim_{k\to\infty} x_k = \lim_{k\to\infty} P(x^*)$.
    Combining the above facts, it follows that $x = P(x^*)$, and so $x^* \in \cF(P)$.
    Therefore, $\cF(P)$ is closed.

    To show that $\cF(P)$ is convex, assume $x, y \in \cF(P)$.
    Suppose $z = \alpha x + (1-\alpha) y$ for some value $\alpha \in (0,1)$.
    Assume, towards a contradiction, that $z \notin \cF(P)$.
    First, using the triangle inequality,
    \[ \| x-y \| = \|x - P(z) + P(z) - y\| \le \| P(z) - x \| + \| P(z) - y \| \]
    Since $z \notin \cF(P)$, $P$ is a paracontraction, and $x, y \in \cF(P)$,
    it follows that both $\| P(z) - x \| < \| z - x \|$ and $\| P(z) - y \| < \| z - y \|$.
    Thus,
    \begin{equation} \label{eqn:convex-1}
        \| x - y \| < \| z - x \| + \| z - y \|
    \end{equation}
    From the definition of $z$,
    \[ \| z - x \| = \| \alpha x + (1-\alpha) y - x \| = (1-\alpha) \| x - y \| \]
    and
    \[ \| z - y \| = \| \alpha x + (1-\alpha) y - y \| = \alpha \| x - y \|. \]
    Consequently,
    \[ \| z - x \| + \| z - y \| = (1-\alpha) \| x - y \| + \alpha \| x - y \| = \| x - y \| \]
    From this and (\ref{eqn:convex-1}), it follows that $\| x - y \| < \| x - y \|$, which is a contradiction.
    Therefore $z \in \cF(P)$ and so $\cF(P)$ is convex.
\end{proof-of}
\begin{proof-of}{\autoref{linear-qne-pc-iff}}
    To prove the first claim,
    assume $P$ is quasi-nonexpansive with respect to $\|\cdot\|$.
    Suppose $x$ and $y$ are vectors in $\R^n$.
    By linearity of $P$, it follows that $\| P(x) - P(y) \| = \| P(x-y) - 0 \|$.
    Additionally, $0 = P(0)$, or in other words, $0 \in \cF(P)$.
    Thus, $\|P(x-y) - 0\| \le \| x - y - 0 \| = \|x - y\|$, using the assumption that $P$ is quasi-nonexpansive.
    So $\|P(x) - P(y)\| \le \| x -y \|$ and therefore $P$ is nonexpansive.

    Next, assume $P$ is nonexpansive with respect to $\|\cdot\|$.
    Suppose $x \in \R^n$ and $y \in \cF(P)$.
    Since $y$ is a fixed point of $P$, $\|P(x) - y\| = \|P(x) - P(y)\|$.
    But $P$ is nonexpansive, therefore $\|P(x) - P(y)\| \le \|x - y\|$.
    Thus, $\|P(x) - y\| \le \|x - y\|$ and so $P$ is quasi-nonexpansive.

    To prove the second claim,
    assume $P$ is a paracontraction with respect to $\|\cdot\|$.
    Suppose $x \notin \cF(P)$.
    Since $P$ is linear, $0 = P(0)$, or in other words, $0 \in \cF(P)$.
    Thus, $\|P(x)\| = \|P(x) - 0\| < \| x - 0 \| = \|x\|$, using the assumption that $P$ is a paracontraction.

    Next, assume $\|P(x)\| < \|x\|$ for any $x \notin \cF(P)$.
    Since $P$ is linear and $\|\cdot\|$ is a norm, $P$ is continuous as well.
    Note that for any vector $y \notin \cF(S\otimes I)$,
    $x-y \notin \cF(S\otimes I)$ because $x \notin \cF(P)$.
    Since $\|P(x)\| < \|x\|$ holds for all vectors which are not fixed points of $P$,
    it must be true that $\|P(x) - y\| < \|x-y\|$ so $P$ is a paracontraction.
\end{proof-of}

\end{document}